\documentclass[12pt]{amsart}
\usepackage{amssymb}
\newcommand{\h}{\mathfrak{h}}
\newcommand{\C}{\mathbb{C}}
\newcommand{\He}{\mathcal{H}}
\newcommand{\KZ}{\operatorname{KZ}}
\newcommand{\OCat}{\mathcal{O}}

\newtheorem{Thm}{Theorem}[section]

\newtheorem{Lem}[Thm]{Lemma}
\theoremstyle{definition}

\numberwithin{equation}{section}
\address{Department
of Mathematics, Northeastern University, Boston MA 02115 USA}
\email{i.loseu@neu.edu}
\title{Finite dimensional quotients of Hecke algebras}
\author{Ivan Losev}
\begin{document}
\begin{abstract}
Let $W$ be a complex reflection group. We prove that there is the maximal finite dimensional quotient of
the Hecke algebra $\mathcal{H}_q(W)$ of $W$ and that the dimension of this quotient coincides with
$|W|$. This is a weak version of a Brou\'{e}-Malle-Rouquier conjecture from 1998. The proof is based
on categories $\mathcal{O}$ for Rational Cherednik algebras.
\end{abstract}
\maketitle
\section{Introduction}
Let $W$ be a complex reflection group. Recall that such groups are fully classified, see \cite{ST}.
In this context, one can also define the braid group $B_W$. Namely, let $\h$ denote the reflection
representation of $W$. Inside $\h$, one considers the open subset of regular points $\h^{reg}:=\{x\in \h| W_x=\{1\}\}$
so that $W$ is the Galois group of the cover $\h^{reg}\twoheadrightarrow \h^{reg}/W$. By definition,
the braid group $B_W$ is the fundamental group $\pi_1(\h^{reg}/W)$.

If $W$ is a Coxeter group, one considers a flat deformation of $\C W$ called a {\it Hecke algebra}.
These algebras are of importance in Representation theory (e.g., that of finite groups of Lie type)
and beyond (e.g., in Knot theory). One can define Hecke algebras for complex reflection groups as well.
In the general case this was done in \cite[4C]{BMR}. To recall the definition we need some more
notation. Namely, let $\mathfrak{H}$ denote the set of the reflection hyperplanes for $W$.
For $\Gamma\in \mathfrak{H}$, let $W_\Gamma$ denote the pointwise stabilizer of $\Gamma$, this is a cyclic group.
Set $\ell_\Gamma:=|W_\Gamma|$. The group $B_W$ is generated by elements $T_\Gamma, \Gamma\in \mathfrak{H},$ where,
roughly speaking, $T_\Gamma$ the rotation around $\Gamma$ by $2\pi/\ell_\Gamma$, see
\cite[Section 2]{BMR}.
Now pick independent variables $u_{\Gamma,i}, i=0,1,\ldots,\ell_{\Gamma}-1$ with $u_{\Gamma,i}=u_{\Gamma',i}$
for $W$-conjugate $\Gamma,\Gamma'$. Set ${\bf u}:=(u_{\Gamma,i})$. By definition, \cite[Definition 4.21]{BMR}, the Hecke
algebra $\mathcal{H}_{\bf u}(W)$ is the quotient of $\mathbb{Z}[{\bf u}^{\pm 1}]B_W$ modulo
the relations
$$\prod_{i=0}^{\ell_\Gamma-1}(T_\Gamma-u_{\Gamma,i})=0.$$
Brou\'{e}, Malle and Rouquier conjectured in \cite[Section 4C]{BMR} that $\He_{\bf u}(W)$ is a
free $\mathbb{Z}[{\bf u}^{\pm 1}]$-module generated by $|W|$ elements. Currently, the proof is missing
for several exceptional complex reflection groups. In this paper, we are going
to prove a weaker version of this conjecture.

First of all, we are dealing with specializations to $\C$. For a collection of nonzero
complex numbers $(q_{\Gamma,i})$, where $\Gamma\in \mathfrak{H}/W, i\in \{0,1,\ldots,\ell_\Gamma-1\}$,
consider the algebra $\mathcal{H}_q(W)$, the specialization of $\mathcal{H}_{\bf u}(W)$
with $u_{\Gamma,i}\mapsto q_{\Gamma,i}$. Note that replacing the collection $q_{\Gamma,i}$
with $(\alpha_\Gamma q_{\Gamma,i})$ for $\alpha_\Gamma\in \C^\times$, we get isomorphic algebras, see, e.g., \cite[3.3.3]{rouqqsch}.
So the number of parameters actually equals $|S/W|$, where $S$ stands for the set of complex
reflections in $W$. Note that if $q_{\Gamma,j}=\exp(2\pi\sqrt{-1}j/\ell_\Gamma)$ for all $\Gamma$
and $j$, we just have $\He_q(W)=\C W$. In general, however, it is even unclear whether the algebra
$\He_q(W)$ is finite dimensional or not.  In a way, the infinite dimension is the only obstruction
to $\dim \He_q(W)=|W|$. More precisely, we have the following theorem that is the main
result of this paper.

\begin{Thm}\label{Thm:main}
There is a  minimal two-sided ideal $I\subset \mathcal{H}_q(W)$ such that $\mathcal{H}_q(W)/I$
is finite dimensional. Moreover, we have $\dim \mathcal{H}_q(W)/I=|W|$.
\end{Thm}

Other results towards the Brou\'{e}-Malle-Rouquier conjecture were known before, see \cite{Marin}
for a review. One advantage of our approach is that it is fully conceptual and does not involve
any case-by-case arguments.

The key idea of the proof is to use categories $\mathcal{O}$ for Rational Cherednik algebras
$H_c(W)$ introduced in \cite{GGOR}. By definition,  the algebra $H_c(W)$ is the subalgebra
in the skew-group algebra $D(\h^{reg})\#W$ generated by $\C[\h],\C W$ and so called Dunkl
operators $D_a, a\in \h$. Those are differential operators with first order poles along
the reflection hyperplanes $\Gamma$. We have a triangular decomposition $H_c(W)=\C[\h]\otimes \C W\otimes
S(\h)$ that allows to define the category $\mathcal{O}$. This is the category of all $H_c(W)$-modules
that are finitely generated over $\C[\h]$ and have a locally nilpotent action of $\h(\subset S(\h))$.

Let us pick $M\in \mathcal{O}$. Its restriction to $\h^{reg}$ is a $W$-equivariant local system
on $\h^{reg}$. So the fiber $M_x$ carries a monodromy representation of $B_W$. It was shown
in \cite{GGOR} that the $\C B_W$-action on $M_x$ factors though a certain quotient of $\He_q(W)$
that has dimension $|W|$. We will show that every finite dimensional $\He_q(W)$-module can be
represented in the form $M_x$ for some $M\in \mathcal{O}$.

The assumption $\dim\mathcal{H}_q(W)=|W|$ is actually important for the representation theory of
$H_c(W)$. Theorem \ref{Thm:main}
should make it possible to remove this assumption but we are not going to elaborate on that.

The paper is organized as follows. In Section \ref{SS_gen} we gather various generalities on the
Rational Cherednik algebras and their categories $\mathcal{O}$. Then in Section \ref{SS_proof}
we prove the main theorem.

{\bf Acknowledgements}. This project was sparked by a discussion with Raphael Rouquier in April 2014.
I would like to thank him and  Roman Bezrukavnikov, Pavel Etingof and Dennis Gaitsgory for stimulating discussions.
Besides, I am grateful to Pavel  Etingof for his comments on a preliminary version of this paper
and for pointing out the reference \cite{Wilcox}.
My work was supported by the NSF  under Grant  DMS-1161584.

\section{Generalities}\label{SS_gen}
\subsection{Rational Cherednik algebras}\label{SS_RCA}
Rational Cherednik algebras were introduced by Etingof and Ginzburg in \cite{EG}. In this subsection
we recall the definition.

Let $W$ be a complex reflection group and $\h$ be its reflection representation. By $S$ we denote the subset of
$W$ consisting of the complex reflections. For $s\in S$, pick an eigen-vector $\alpha_s\in \h^*$
for $s$ with eigenvalue $\lambda_s\neq 1$. We fix a $W$-invariant function $c:S\rightarrow \C$.
Using this function, for $a\in \h$, we can define the Dunkl operator $D_a\in D(\h^{reg})\#W$ by the following
formula:
$$D_a=\partial_a+\sum_{s\in S} \frac{2c(s)}{1-\lambda_s}\frac{\langle \alpha_s,a\rangle}{\alpha_s}(s-1).$$
The Rational Cherednik algebra $H_c(W)$ is the subalgebra in $D(\h^{reg})\#W$ generated by $\C[\h],\C W$
and the Dunkl operators $D_a, a\in \h$. Alternatively, one can present $H_c(W)$ by generators and relations:
$H_c(W)$ is the quotient of $T(\h\oplus \h^*)\#W$ by the relations
$$[x,x']=[y,y']=0, [y,x]=\langle y,x\rangle-\sum_{s\in S} c(s)\langle \alpha_s,y\rangle\langle \alpha_s^\vee,x\rangle s, \quad x,x'\in \h^*, y,y'\in \h.$$
Here we write $\alpha_s^\vee$ for the eigenvector of $s$ in $\h$ with eigenvalue $\lambda_s^{-1}$ and $\langle \alpha_s^\vee,\alpha_s\rangle=2$. To get from the second definition to the first one, we use a homomorphism
$H_c(W)\rightarrow D(\h^{reg})\#W$ given by $x\mapsto x, w\mapsto w, y\mapsto D_y$. Set $\delta:=\prod_{s\in S}\alpha_s^{\ell_s}$, where $\ell_s$ stands for the order of $s$ (note that this is slightly different from the usual
definition). This is a $W$-invariant element and
the operator $\operatorname{ad}\delta:H_c(W)\rightarrow H_c(W)$ is locally nilpotent so the localization
$H_c(W)[\delta^{-1}]$ is defined. The homomorphism
$H_c(W)\rightarrow D(\h^{reg})\#W$ extends to an isomorphism $H_c(W)[\delta^{-1}]\xrightarrow{\sim} D(\h^{reg})\#W$.

The algebra $H_c(W)$ admits a triangular decomposition: a natural map $S(\h^*)\otimes \C W\otimes S(\h)\rightarrow
H_c(W)$ is an isomorphism. Also $H_c(W)$ is graded with $\deg x=1, \deg w=0, \deg y=-1, x\in \h^*, w\in W, y\in \h$.
We call this grading the Euler grading. It is inner: it is given by the eigenvalues of $\operatorname{ad}h$,
where
\begin{equation}\label{eq:Euler}
h=\sum_{i=1}^n x_i y_i+\frac{n}{2}-\sum_{s\in S} \frac{2c(s)}{1-\lambda_s}s.\end{equation}

Now let us discuss the base change for $H_c(W)$. Let $U$ be an affine algebraic variety
equipped with an \'{e}tale map $U\rightarrow \h/W$. Then $\C[U]\otimes_{\C[\h]^W}H_c(W)$
has a natural algebra structure, it is a subalgebra in $D(U\times_{\h/W}\h^{reg})\# W$
generated by $\C[U]\otimes_{\C[\h/W]}\C[\h],\C W$ and the Dunkl operators. Similarly,
if $U$ is a Stein complex analytic manifold (again equipped with an \'{e}tale map $U\rightarrow \h/W$), then $\C_{an}[U]\otimes_{\C[\h/W]}H_c(W)$
is an algebra. Here and below $\C_{an}[U]$ denotes the algebra of analytic functions on $U$.

\subsection{Categories $\mathcal{O}$ and $\KZ$ functor}
The category $\mathcal{O}$ for $H_c(W)$ was defined in \cite{GGOR}. By definition, it consists of all
$H_c(W)$-modules $M$ that are finitely generated over $S(\h^*)=\C[\h]$ and where the action of $\h$
is locally nilpotent. Equivalently, $\mathcal{O}$ consists of all $H_c(W)$-modules $M$ that are finitely
generated over $\C[\h]$ and can be graded. This category $\mathcal{O}$ will be denoted by $\OCat_c(W)$.

Let us proceed to the KZ functor introduced in \cite[Section 5]{GGOR}. Pick $M\in \OCat_c(W)$.
Then $M[\delta^{-1}]$ is a $W$-equivariant local system on $D(\h^{reg})$ with regular singularities.
The category of such local systems is equivalent to the category $B_W\operatorname{-mod}_{fin}$
of the finite dimensional $B_W$-modules, to a local system $M'$ one assigns its fiber (or, more precisely,
the fiber of its descent to $\h^{reg}/W$) equipped
with the monodromy representation. It turns out that the monodromy representation associated to
$M[\delta^{-1}]$ factors through $\mathcal{H}_q(W)$, \cite[5.3]{GGOR}, where the parameter $q$ is computed as
follows.  For a reflection hyperplane $\Gamma$, set
\begin{align}\label{eq:h_to_c} &h_{\Gamma,i}=\frac{1}{\ell_\Gamma}\sum_{s\in W_\Gamma\setminus \{1\}}\frac{2c(s)}{\lambda_s-1}\lambda_s^{-i},\\\label{eq:q_to_h}
&q_{\Gamma,i}=\exp(2\pi\sqrt{-1}(h_{\Gamma,j}+j/\ell_H)).
\end{align}
So we get an exact functor $\KZ:\OCat_c(W)\rightarrow \He_q(W)\operatorname{-mod}_{fin}$.
This functor is given by $\operatorname{Hom}_{\OCat_c(W)}(P_{\KZ},\bullet)$, where
$P_{KZ}$ is a projective object such that $\dim \operatorname{End}_{\OCat_c(W)}(P_{\KZ})=|W|$
equipped with a homomorphism $\He_q(W)\rightarrow \operatorname{End}_{\OCat_c(W)}(P_{\KZ})$.
The proof of \cite[Theorem 5.15]{GGOR} shows that this homomorphism is surjective.

Theorem \ref{Thm:main} will follow if we show that the functor $\KZ$
is essentially surjective.

\subsection{Isomorphisms of \'{e}tale lifts}\label{SS_lift_iso}
Here we are going to recall some results of \cite{BE} regarding isomorphisms of completions.

Let $W'\subset W$ be a parabolic subgroup, i.e., the stabilizer of a point in $\h$.
Set $\h^{reg-W'}:=\{b\in \h| W_b\subset W'\}$. The complement of $\h^{reg-W'}$ in $\h$ is
the union of the hyperplanes $\ker\alpha_s$ for $s\not\in W'$. So $\h^{reg-W'}$
is a principal open subset of $\h$. Note that the natural morphism $\h^{reg-W'}/W'\rightarrow \h/W$
is \'{e}tale (and $\h^{reg-W'}/W'$ is precisely the unramified locus of $\h/W'\rightarrow \h/W$).

Consider the space $H_c(W)_{reg-W'}:=\C[\h^{reg-W'}]^{W'}\otimes_{\C[\h]^W}H_c(W)$. As was mentioned in
 the end of Subsection \ref{SS_RCA},  $H_c(W)_{reg-W'}$ is actually an algebra.
In \cite[3.3]{BE}, Bezrukavnikov and Etingof essentially found an alternative description
of this algebra. Namely, consider the Cherednik algebra $H_c(W',\h)$ defined for the pair
$W',\h$, it decomposes into the tensor product $H_c(W',\h)=D(\h^{W'})\otimes H_c(W')$
(here we abuse the notation and write $c$ for the restriction of $c$ to $S\cap W'$; $D(\h^{W'})$
stands for the algebra of differential operators on $\h^{W'}$)
and form its localization $H_c(W,\h)_{reg-W'}:=\C[\h^{reg-W'}]^{W'}\otimes_{\C[\h]^{W'}}H_c(W',\h)$.
Then, following \cite[3.2]{BE}, we can form the centralizer algebra $Z(W,W', H_c(W',\h)_{reg-W'})$.
Recall that, by definition, for an algebra $A$ equipped with a homomorphism $\C W'\rightarrow A$
one defines the centralizer algebra $Z(W,W',A)$ as the endomorphism algebra of the right $A$-module
$\operatorname{Map}_{W'}(W,A)=\{f:W\rightarrow A| f(w'w)=w'f(w)\}$. Choosing representatives of the
left $W'$-cosets in $W$, we get an identification $Z(W,W',A)\cong \operatorname{Mat}_{|W/W'|}(A)$.
The algebra $A$ can be recovered back from $Z(W,W',A)$ as follows. Consider the element $e(W')\in Z(W,W',A)$
given by $e(W')f(u)=f(u)$ if $u\in W'$ and $0$ else. Then $e(W')Z(W,W',A)e(W')$ is naturally identified with
$A$.

The following is essentially \cite[Theorem 3.2]{BE} (they considered completions instead of
\'{e}tale lifts but the proof works in our situation as well).

\begin{Lem}\label{Lem:isom}
There is a unique isomorphism
\begin{equation}\label{eq:theta_iso}\theta: H_c(W)_{reg-W'}\xrightarrow{\sim} Z(W,W', H_c(W',\h)_{reg-W'})\end{equation} such that
the following hold for any $f\in \operatorname{Map}_{W'}(W,A)$ and any $u\in W$
\begin{align*}
& [\theta(g)f](u)=gf(u),\quad g\in \C[\h^{reg-W'}]^{W'},\\
& [\theta(\alpha)f](u)= (u\alpha) f(u), \quad \alpha\in \h^*,\\
& [\theta(w)f](u)=f(uw), \quad w\in W,\\
& [\theta(a)f](u)=(ua)f(u)+\sum_{s\in S\setminus W'}\frac{2c(s)}{1-\lambda_s}\frac{\langle \alpha_s,ua\rangle}{\alpha_s}(f(su)-f(u)), \quad a\in \h.
\end{align*}
\end{Lem}

Note that the algebras in (\ref{eq:theta_iso})
come equipped with $\C^\times$-actions by algebra automorphisms. For example, the action of $H_c(W)_{reg-W'}$ comes
from the action on $H_c(W)$ induced from the Euler grading and the action on $\h^{reg-W'}/W'$ induced from
the scaling action ($t.x=t^{-1}x$)  on $\h$. The isomorphism $\theta$ is $\C^\times$-equivariant.

We can further restrict $\theta$ to some analytic submanifolds or formal subschemes of $\h^{reg-W'}/W'$.
Choose a little disk $Y\subset \h^{reg-W'}\cap \h^{W'}$ and also a little disk $D$ around $0$ in
$\h_{W'}/W'$, where $\h_{W'}$ denotes a unique $W'$-stable complement to $\h^{W'}$ in $\h$.
We set $\widehat{Y}:=Y\times D$, this is an open submanifold in $\h^{W'-reg}/W'$, in $(\h^{reg-W'}\cap \h^{W'})\times
\h_{W'}/W'$ or in $\h/W$ (under the natural morphism $\h^{reg-W'}/W'\rightarrow \h/W$).

The algebra $\C_{an}[\widehat{Y}]\otimes_{\C[\h/W']}H_c(W',\h)$ coincides with
$$\C_{an}[\widehat{Y}]\otimes_{\C[\h_{W'}]\otimes \C[Y]}(H_c(W')\otimes D(Y))$$
so we get an isomorphism
\begin{equation}\label{eq:isom_Y}
\begin{split}
&\theta_Y: \C_{an}[\widehat{Y}]\otimes_{\C[\h/W]}H_c(W)\xrightarrow{\sim} \\ &Z(W,W',\C_{an}[\widehat{Y}]\otimes_{\C[\h_{W'}]\otimes \C[Y]}(H_c(W')\otimes D(Y))).
\end{split}
\end{equation}
Note that this isomorphism is compatible with the Euler derivations.

We can restrict even further. Pick a point $b\in Y$ and consider the completion $\C[\h/W]^{\wedge_b}$
of $\C[\h/W]$ with respect to the maximal ideal defined by $b$. Then $\theta_Y$ induces
\begin{equation}\label{eq:isom_b}
\theta_b: \C[\h/W]^{\wedge_b}\otimes_{\C[\h/W]}H_c(W)\xrightarrow{\sim}
Z(W,W', \C[\h/W']^{\wedge_b}\otimes_{\C[\h/W']}H_c(W',\h)).
\end{equation}
This is an isomorphism originally constructed in \cite{BE}.

\section{Proof of the main theorem}\label{SS_proof}
\subsection{Scheme of the proof}
Let $V$ be a finite dimensional $\He_q(W)$-module and let $N$ denote the corresponding $W$-equivariant $D$-module
on $\h^{reg}$. Our goal is to show that there is $M\in \OCat_c(W)$ such that
$M[\delta^{-1}]\cong N$. This consists of two  steps.
\begin{itemize}
\item[(I)] Set $\h^{sr}:=\{b\in W| \dim\h^{W_b}\geqslant \dim \h-1\}$. This is an open subset
that coincides with $\bigcup_{\Gamma\in \mathfrak{H}} \h^{reg-W_\Gamma}$, the  codimension
of its complement is bigger than $1$. We will see that there is an $\mathcal{O}_{\h^{sr}/W}$-coherent
$\mathcal{O}_{\h^{sr}/W}\otimes_{\C[\h]^W}H_c(W)$-module $\tilde{N}$
whose restriction to $\h^{reg}/W$ is isomorphic to $N$ and that carries a locally finite
derivation compatible with the Euler derivation of $H_c(W)$.
\item[(II)]  We will see that $\tilde{N}$ is a vector bundle.
From here we will deduce that the global sections of $\tilde{N}$ are finitely generated and
hence lie in $\mathcal{O}_c(W)$. Then we take $M:=\Gamma(\tilde{N})$.
\end{itemize}

Let us elaborate on how we are going to achieve (I).
First, in Subsection \ref{SS_weak_equi} we will check that the Euler vector field acts on $N$
locally finitely. This will eventually prove that $\tilde{N}$ comes equipped with a locally
finite derivation that is compatible with the Euler one on $H_c(W)$.

Now let us explain how we produce $\tilde{N}$, this is done in Subsection \ref{SS_codim1}.
Take $W'=W_\Gamma$ and let $\widehat{Y}$ have the same
meaning as in Subsection \ref{SS_lift_iso}. 
Set $\widehat{Y}^\times:=\widehat{Y}\setminus Y$.  Consider $N_\Gamma:=e(W')(\C_{an}[\widehat{Y}^\times]\otimes_{\C[\h^{reg}/W]}N)$. This is a vector bundle
on $\widehat{Y}^\times$ with a meromorphic connection that has pole of order $1$ on $Y$
(since $N_\Gamma$ is obtained by restricting an algebraic vector bundle
$e(W')\eta_\Gamma^*N$, where $\eta_\Gamma$ is a natural
morphism $\h/W_{\Gamma}\rightarrow \h/W$,
it makes sense to speak about sections
of $N_\Gamma$ with poles on $Y$; here and below $\eta_\Gamma$ denotes the projection $\h/W'\rightarrow \h/W$).
Our first step will be to see that $N_\Gamma$
is obtained by restricting a $\C_{an}[\widehat{Y}]\otimes_{\C[\h/W]}H_c(W)$-module $M_\Gamma$. Then we will see that
$[e(W')\eta_{\Gamma}^*N]\cap M_\Gamma$ (the intersection of subspaces in $N_\Gamma$) is finitely generated
over $\C[\h^{reg-W'}/W']$. We will get $\tilde{N}$, roughly speaking, by taking the intersection
of $N$ and $\eta_{\Gamma}^* N\cap M_\Gamma$ over all possible $\Gamma$.\footnote{After this paper was written,
I have learned from Etingof
that most of the proof is already contained in some form in \cite{Wilcox}. Lemmas 5.7,5.8 of {\it loc.cit.}
are similar to what is done in Subsection \ref{SS_codim1}, while the main result of Subsection
\ref{SS_glob_sec} has a somewhat easier proof, \cite[Lemma 3.6]{Wilcox}.}

\subsection{Locally finite derivation}\label{SS_weak_equi}
Our goal here is to show that the Euler vector field acts on $N$ locally finitely. Recall that $N$
is a local system on $\h^{reg}/W$ with regular singularities. Our claim is a consequence of the following
general result.

\begin{Lem}\label{Lem:weak_equiv}
Let $X$ be the complement to a $\C^\times$-stable divisor in $\C^d$ and let $N$ be a local system with regular
singularities on $X$. Then the Euler vector field $\mathsf{eu}$ acts on $N$ locally finitely, meaning that
every $n\in N$ is included into a finite dimensional $\mathsf{eu}$-stable subspace.
\end{Lem}
This claim should be standard, but we provide the proof in the sake of completeness. For a different proof,
see \cite[Lemma 3.2]{Wilcox}.
\begin{proof}
If $N'\subset N$ is a $D(X)$-submodule and $\mathsf{eu}$ acts locally finitely on $N/N',N'$, then
the same is true for $N$. So it is enough to assume that $N$ (and hence $V$) is irreducible.
Consider the element $\eta\in \pi_1(X)$ given by the loop $\exp(2\pi \sqrt{-1}t)x_0, t\in [0,1]$, where
$x_0$ denotes the base point.
The element $\eta$ is central and hence has to act on $V$ by a scalar. Under the Riemann-Hilbert correspondence,
the latter translates to the claim that $N$ is twisted equivariant with respect to the $\C^\times$-action.
This implies our claim.
\end{proof}

\subsection{Extension to codimension $1$}\label{SS_codim1}
We start by producing $M_\Gamma$.

Set $\widehat{Y}^\times:=Y\times D^\times=\widehat{Y}\setminus Y$. Consider the category
$\operatorname{Loc}_{rs}(\widehat{Y},Y)$ of meromorphic local system on $\widehat{Y}^\times$
with regular singularities on $Y$ (so an object in $\operatorname{Loc}_{rs}(\widehat{Y},Y)$ comes
equipped with a lattice over the ring of meromorphic differential operators on $\widehat{Y}^{\times}$,
a morphism in the category is supposed to preserve this lattice).
The category $\operatorname{Loc}_{rs}(\widehat{Y},Y)$
is equivalent to the category $\C[T^{\pm 1}]\operatorname{-mod}$
of the finite dimensional $\C[T^{\pm 1}]$-modules  via taking the monodromy
representation because $\pi_1(\widehat{Y}^\times)=\pi_1(D^\times)=\mathbb{Z}$
(here we use the regular singularities condition).
Under the equivalence $\C[T^{\pm 1}]\operatorname{-mod}\cong \operatorname{Loc}_{rs}(\widehat{Y},Y)$,
the KZ functor becomes $$\OCat_c(W')\rightarrow \operatorname{Loc}_{rs}(\widehat{Y},Y), M\mapsto \C_{an}[\widehat{Y}^\times]\otimes_{\C[\h_{W'}/W']}M,$$
(with meromorphic lattice $\C_{an}[\widehat{Y}][\nu^{-1}]\otimes_{\C[\h_{W'}/W']}M$, where $\nu$ denotes a coordinate
on $\h_{W'}/W'$).
The right adjoint $\operatorname{KZ}^*$ sends $N'\in \operatorname{Loc}_{rs}(\widehat{Y},Y)$
to the subspace of $N'$ of all meromorphic elements that are annihilated by the vector fields on $Y$ and lie in
the generalized eigenspace for $\h_{W'}\subset H_c(W')$ with eigenvalue $0$.

Now we can produce a $\C[\widehat{Y}]\otimes_{\C[\h/W']}H_c(W',\h)$-module $M_\Gamma\in \OCat_c(W',\widehat{Y})$. Set $N_\Gamma= e(W')(\C_{an}[\widehat{Y}^\times]\otimes_{\C[\h^{reg}/W]}N)$. This is an object in $\operatorname{Loc}_{rs}(\widehat{Y},Y)$.
Note that, under the equivalence $\operatorname{Loc}_{rs}(\widehat{Y},Y)\cong \C[T^{\pm 1}]\operatorname{-mod}$, we have $N_\Gamma\in \He_q(W')\operatorname{-mod}$. Now set  
\begin{equation}\label{eq:M_G_descr} M_{\Gamma}:=\C_{an}[\widehat{Y}]\otimes_{\C[\h_{W'}/W']}\operatorname{KZ}^*(N_\Gamma).\end{equation}
Note that the description of $\operatorname{KZ}^*(N_\Gamma)\subset N_\Gamma$ implies that it is stable under the Euler
vector field on $N_\Gamma$. So $M_{\Gamma}\subset N_{\Gamma}$ is also stable under the Euler vector field.

Set $\tilde{N}_\Gamma:= M_\Gamma\cap e(W')\eta_\Gamma^* N$ (the intersection is taken inside
$N_\Gamma$). This is a  submodule in
the $\C[\h^{reg-W'}]^{W'}\otimes_{\C[\h/W']}H_c(W')$-module $e(W')\eta_{\Gamma}^* N$ stable under the Euler
vector field.

\begin{Lem}\label{Lem:fin_gen}
The module $\tilde{N}_\Gamma$ is finitely generated over $\C[\h^{reg-W'}]^{W'}$ and satisfies
$\tilde{N}_{\Gamma}[\nu^{-1}]=e(W')\eta_{\Gamma}^* N$, where $\nu$ is a coordinate on $\h_{W'}/W'\cong \C$.
\end{Lem}
\begin{proof}
Note that the epimorphism $\He_q(W')\rightarrow \operatorname{End}_{\OCat_c(W')}(P_{KZ})^{opp}$
is an isomorphism.
Let $N_{\Gamma}^{mer}\subset N_{\Gamma}$ denote the meromorphic lattice.
Then $M_\Gamma$ is contained in $N_{\Gamma}^{mer}$ and is a $\C_{an}[\widehat{Y}]$-lattice
there. 
Indeed, it is enough to show this fiberwise (i.e., at any point of $Y$), where this is clear (to prove
that $M_{\Gamma}$ is a lattice we use the observation that $N_\Gamma$ is the image of $M_{\Gamma}$
under the KZ functor). For any other lattice $M'$
we have $\nu^{d}M'\subset M_{\Gamma}\subset \nu^{-d} M'$ for some $d>0$. So it is enough to show that $N\cap M'$ is finitely generated for some lattice $M'$. Let us produce
such $M'$.

As a $\C[\h^{reg}/W]$-module, $N$ is projective and hence is a direct summand in a free module, say $\C[\h^{reg}/W]^{\oplus r}$.
So $N_{\Gamma}^{mer}$ is a direct summand in $\C_{an}[\widehat{Y}][\nu^{-1}]^{\oplus r}$. The intersection
$M':=N_{\Gamma}^{mer}\cap \C_{an}[\widehat{Y}]^{\oplus r}$ (inside $\C_{an}[\widehat{Y}][\nu^{-1}]^{\oplus r}$) is clearly a lattice in $N_{\Gamma}^{mer}$. Further,
the intersection $M'\cap e(W')\eta_\Gamma^*N$ coincides with $e(W')\eta_{\Gamma}^* N\cap \C[\h^{reg-W'}/W']^{\oplus r}$
and hence is finitely generated (and clearly is a lattice in $e(W')\eta_{\Gamma}^* N$). We are done.
\end{proof}

Now we are ready to define a module $\tilde{N}$ over $\mathcal{O}_{\h^{sr}/W}\otimes_{\C[\h/W]}H_c(W)$. Abusing the notation, let us write $\tilde{N}_\Gamma$ for the corresponding (under the equivalence) $\C[\h^{reg-W_{\Gamma}}/W_{\Gamma}]\otimes_{\C[\h/W]}H_c(W)$-module, the restriction of $\tilde{N}_\Gamma$
to $\eta_\Gamma^{-1}(\h^{reg}/W)$ coincides with $\eta_\Gamma^* N$ by the construction. Let  $\iota_\Gamma: \h^{reg-W_\Gamma}\hookrightarrow \h$ be the inclusion and $\pi_\Gamma:\h^{reg-W_\Gamma}\twoheadrightarrow \h^{reg-W_{\Gamma}}/W_{\Gamma}$
be the quotient morphism. Also let $\pi: \h\rightarrow \h/W$ denote the quotient morphism
and $\iota:\h^{reg}\hookrightarrow \h^{sr}$ be the inclusion.
Note that, by the construction, $\iota_{\Gamma*}\pi_{\Gamma}^* \tilde{N}_{\Gamma}\subset \iota_*\pi^* N$
(recall that we  view $N$ as a coherent sheaf on $\h^{reg}/W$). The intersection
$\widehat{N}:=\bigcap_{\Gamma} \iota_{\Gamma*}\pi_{\Gamma}^* \tilde{N}_{\Gamma}$ is a coherent sheaf on $\h^{sr}$
because of Lemma \ref{Lem:fin_gen} and the equality
$\h^{sr}=\bigcup_\Gamma\h^{reg-W_{\Gamma}}$. The intersection is stable under the Euler vector field
because all $\tilde{N}_{\Gamma}$ are.
Also $\widehat{N}$ is $W$-stable, this is because  $w\pi_{\Gamma}^* \tilde{N}_{\Gamma}=\pi_{w\Gamma}^* \tilde{N}_{w\Gamma}$. Now set $\tilde{N}:=\pi_*(\widehat{N})^W=\pi_*\widehat{N}\cap \iota'_* N$
(where $\iota':\h^{sr}/W\hookrightarrow \h/W$ denotes the inclusion). This is a coherent sheaf on
$\h^{sr}/W$
stable under the Euler vector field on $\iota'_* N$.  It remains to show that $\tilde{N}\subset \iota'_* N$
is stable under $H_c(W)$. But this follows from
\begin{equation}\label{eq:tild_alt}\tilde{N}=\iota'_* N\cap\bigcap_\Gamma \eta_{\Gamma*} \tilde{N}_{\Gamma},\end{equation}
where now we view $\eta_\Gamma$ as a morphism $\h^{reg-W'}/W'\rightarrow \h^{sr}/W$. (\ref{eq:tild_alt})
follows from the observation that $\tilde{N}_\Gamma=\pi_{\Gamma*}(\pi_{\Gamma}^* \tilde{N}_\Gamma)^{W_\Gamma}$
and $\pi=\eta_{\Gamma}\circ \pi_{\Gamma}$. Since all sheaves in the right hand side of (\ref{eq:tild_alt})
are stable under $H_c(W)$, we see that $\tilde{N}$ is stable as well. It follows from the construction
that $\tilde{N}|_{\h^{reg}/W}\cong N$.

\subsection{Global sections}\label{SS_glob_sec}
\begin{Lem}\label{Lem:vect_bundle}
The sheaf $\tilde{N}$ is a vector bundle on $\h^{sr}/W$.
\end{Lem}
\begin{proof}
The proof is inspired by \cite[3.2]{EGL}. We need to show that
$\tilde{N}$ is maximal Cohen-Macaulay when viewed as a coherent sheaf on $\h^{sr}/W$. Let $Z$ denote the non-CM
locus of $\tilde{N}$ in $\h^{sr}/W$ and let $d$ be the codimension of $Z$ in $\h^{sr}/W$.
Pick an open $\C^\times$-stable affine subvariety $U$ of $\h^{sr}/W$ that intersects
$Z$ (or, more precisely, an irreducible component of maximal dimension in $Z$).
Consider $H^i_{U\cap Z}(U, \tilde{N})$ for $i<d$. Similarly to \cite[3.2]{EGL}, all these groups
are  $\C[U]\otimes_{\C[\h/W]}H_c(W)$-modules, finitely generated
over $\C[U]$ (this follows from \cite[Expose VIII, Cor. 2.3]{Groth} using equivalence
of (ii) and (iii) there, note that the depth used in (ii) coincides with codimension thanks
to the choice of $Z$). Moreover, by the choice of $Z$, one of these modules is nonzero,
as in \cite[3.2]{EGL}.    The support of $R:=H^i_{U\cap Z}(U, \tilde{N})$ is contained in $Z\cap U$.

Pick $b\in \h$ lying over the support of $R$. Recall
the isomorphism $\theta_b:H_c(W)^{\wedge_b}\cong Z(W,W', H_c(W',\h)^{\wedge_b})$, where we take $W'=W_b$.
So we get  a nonzero $H_c(W',\h)^{\wedge_b}$-module $e(W')\theta_{b*}(R^{\wedge_b})$. This module
is finitely generated over $\C[\h]^{\wedge_b}$. So it is of the form $\C[\h_{W'}]^{\wedge_b}\otimes R_0^{\wedge_0}$
for $R_0\in \OCat_c(W')$. It follows that $d=1$ and that $R=\Gamma_{Z\cap U}(U,\tilde{N})$.
But, by the construction, $\Gamma(U,\tilde{N})$ is embedded into $\Gamma(U\cap \h^{reg}/W,N)$
and so $\Gamma(U,\tilde{N})$ has no torsion $\C[U]$-submodules. We get a contradiction
that shows that $\tilde{N}$ is Cohen-Macaulay. Since $\tilde{N}$ is torsion-free, we see that
it is maximal Cohen-Macaulay and hence is a vector bundle.
\end{proof}

%

Now we can use (iv) in \cite[Expose VIII, Cor. 2.3]{Groth} (applied to an extension of $\tilde{N}$ to a coherent
sheaf on $\h/W$) to see that $M:=\Gamma(\h^{sr}/W, \tilde{N})$ is finitely generated over $\C[\h]^W$.
Let us show that the $H_c(W)$-module $M$ lies in $\OCat_c(W)$. By the construction, $M$ carries
a locally finite derivation compatible with the derivation $\operatorname{ad}h$ of $H_c(W)$.
It follows that $M$ is gradable and hence lies in $\OCat$.
Also, by the construction, $M[\delta^{-1}]=N$. This completes the proof.

\end{document}